\documentclass[11pt]{article}
\usepackage{authblk}
\usepackage[top=1.25in, bottom=1.5in, left=1.5in, right=1.5in]{geometry}

\usepackage{ifthen}
\newboolean{PrintVersion}
\setboolean{PrintVersion}{false} 

\usepackage{amsmath,amssymb,amstext} 
\usepackage[pdftex]{graphicx} 

\usepackage[pdftex,pagebackref=false]{hyperref} 

\usepackage{amsfonts,amsthm,amsmath}
\usepackage{amssymb}
\usepackage{authblk}
\usepackage{algorithm}
\usepackage{algorithmic}
\usepackage{graphicx}
\graphicspath{ {image/} }
\usepackage[displaymath,mathlines]{lineno} 
\usepackage{lipsum} 
\usepackage{tikz}
\usetikzlibrary{decorations.pathreplacing}
\usetikzlibrary{calc,patterns}
\usepackage{bbding}
\usepackage{url}
\usepackage{setspace}
\usepackage{subcaption}
\usepackage{xcolor}
\usepackage{soul}
\usepackage{cite}
\usepackage{longtable}

\usetikzlibrary{arrows.meta}

\usetikzlibrary{arrows,shapes,positioning}
\usetikzlibrary{decorations.markings}
\tikzstyle arrowstyle=[scale=1]
\tikzstyle directed=[postaction={decorate,decoration={markings,
		mark=at position .55 with {\arrow[scale=2, arrowstyle]{stealth}}}}]
\tikzstyle reverse directed=[postaction={decorate,decoration={markings,
		mark=at position .5 with {\arrowreversed[scale=2, arrowstyle]{stealth};}}}]

\newcommand{\zed}{\ensuremath{\mathbb Z}}
\newcommand{\eff}{\ensuremath{\mathbb F}}
\newcommand{\Topstrut}[1]{\rule{0pt}{#1ex}}
\newcommand{\Botstrut}[1]{\rule[-#1ex]{0pt}{0pt}} 

\newtheorem{Theorem}{Theorem}[section]
\newtheorem{Definition}{Definition}[section]
\newtheorem{Example}{Example}[section]
\newtheorem{Remark}{Remark}[section]
\newtheorem{Lemma}[Theorem]{Lemma}
\newtheorem{Corollary}[Theorem]{Corollary}

\title{Asymmetric All-or-nothing Transforms}
\author{Navid Nasr Esfahani}
\author{Douglas R.\ Stinson\thanks{D.R.\ Stinson's research is supported by  NSERC discovery grant RGPIN-03882.
}}
\affil{David R.\ Cheriton School of Computer Science\\ University of Waterloo\\
Waterloo, Ontario, N2L 3G1, Canada
}

\date{\today}

\begin{document}
\maketitle

\begin{abstract}
	In this paper, we initiate a study of \emph{asymmetric all-or-nothing transforms} (or \emph{asymmetric AONTs}). A (symmetric) $t$-all-or-nothing transform is a bijective mapping defined on the set of $s$-tuples over a specified finite alphabet. It is required that  knowledge of all but $t$ outputs leaves any $t$ inputs completely undetermined. There have been numerous  papers developing the theory of AONTs as well as presenting various applications of AONTs in cryptography and information security. 
	
	In this paper, we replace the parameter $t$ by two parameters $t_o$ and $t_i$, where $t_i \leq t_o$. The requirement is that knowledge of all but $t_o$ outputs leaves any $t_i$ inputs completely undetermined. When $t_i < t_o$, we refer to the AONT as \emph{asymmetric}.
	
We give several constructions and bounds for various classes of asymmetric AONTs, especially those with 
$t_i = 1$ or $t_i = 2$. We pay particular attention to  \emph{linear} transforms, where the alphabet is a finite field $\eff_q$ and the mapping is linear.
\end{abstract}

\section{Introduction}
In this paper, we study \emph{asymmetric all-or-nothing transforms}, which 
we define informally as follows.

\begin{Definition}
\label{def1}
Suppose $s$ is a positive integer and $\phi : \Gamma^s \rightarrow \Gamma^s$, where $\Gamma$ is a finite set of size $v$ (called an \emph{alphabet}). Thus $\phi$ is a function that maps an input $s$-tuple 
$\mathbf{x} = (x_1, \dots  , x_s)$ to
output $s$-tuple 
$\mathbf{y} = (y_1, \dots  , y_s)$. 
Suppose $t_i$ and $t_o$ are integers such that $1 \leq t_i \le t_o \le s$.

The function $\phi$ is an \emph{$(t_i,t_o,s,v)$-all-or-nothing transform} (or $(t_i,t_o,s,v)$-AONT) provided that the following properties are
satisfied:
\begin{enumerate}
\item  $\phi$ is a bijection.
\item  If any $s - t_o$ of the $s$ outputs $y_1, \dots , y_s$ are fixed, then the values of any $t_i$ inputs  $x_i$ (for $1 \leq i \leq s$) are completely undetermined.
\end{enumerate}
\end{Definition}

\begin{Remark}
{\rm It is not difficult to see that $t_i \leq t_o$ if a $(t_i,t_o,s,v)$-AONT exists, as follows. If only $t_o$ are outputs are unknown, then the number of possible values taken on by any subset of the inputs is at most $v^{t_o}$. Since a subset of  $t_i$ inputs must be completely undetermined,  we must have $v^{t_i} \leq v^{t_o}$, or $t_i \leq t_o$.}
\end{Remark}

All-or-nothing-transforms (AONTs) were invented in 1997 by Rivest \cite{R}. Rivest's work concerned AONTs that are computationally secure. 
Some early papers on various generalizations of AONTs include \cite{Boyko,CDHKS,Desai}.
Stinson \cite{St} introduced and studied all-or-nothing transforms in the setting of unconditional security.
Further work focussing on the existence of unconditionally secure AONTs can be found in \cite{DES,EGS, ES,ES2,WCJ,ZZWG}. AONTs have had numerous applications in cryptography and information security; see \cite{Phd} for an overview.

Rivest's original definition in \cite{R} corresponded to the special case $t_i= t_o = 1$. Most research since then has involved AONTs where $t_i = t_o = t$ for some positive integer $t$. (Such an AONT is often denoted as
a $(t,s,v)$-AONT in the literature.)
In such an AONT, knowing all but $t$ outputs leaves any $t$ inputs undetermined. 
Here we mainly consider AONTs where $t_i < t_o$. Such an AONT can be thought of \emph{asymmetric} in the sense that the number of missing outputs is greater than the number of inputs about which we are seeking information. In general, asymmetric AONTs are easier to construct than AONTs in which $t_i = t_o$ because the requirements are weaker. 

The first example of asymmetric AONTs in the literature is apparently found in Stinson \cite[\S 2.1]{St}. We present this construction in Example \ref{Ex:even_bastion}.

\begin{Example}
	\label{Ex:even_bastion}
	{\rm 
	For $s$ even, a $(1,2,s,2)$-AONT exists as follows. Given $s$ inputs $x_1, \dots , x_s \in \zed_2$, define
\[
\begin{array}{l}
r = \displaystyle \sum_{i = 1}^s x_i\Botstrut{4}\\
y_i = r + x_i, \text{ for } 1 \leq i \leq s.
\end{array}
\]
This yields the $s$ outputs $y_1, \dots , y_s$. The inverse transformation is computed as
	\[
\begin{array}{l}
r' = \displaystyle \sum_{i = 1}^s y_i\Botstrut{4}\\
x_i = r' + y_i, \text{ for } 1 \leq i \leq s.
\end{array}
\]

	
	Suppose we are given $s-2$  of the $s$ outputs, so two outputs are missing. It is clear that each input depends on $s-1$ outputs: $x_i$ is a function of all the $y_j$'s, except for $y_i$. Thus, if two outputs are missing, then no values can be ruled out for $x_i$. $\blacksquare$
	}
\end{Example}

We note that the construction given in Example \ref{Ex:even_bastion} only works for even $s$ (when $s$ is odd, the mapping is not invertible). A construction for odd values of $s$ will be given later (see Lemma \ref{OddBastion}).

Karame et al.~\cite{BastionAONT} introduced \emph{Bastion}, which is a scheme for securely dispersing a document over multiple cloud storage services. Bastion involves encrypting a plaintext using counter mode and then applying a $(1,2,s,2)$-AONT to the resulting ciphertext blocks. The paper \cite{BastionAONT}  considered a threat model where the adversary may have access to the key or use a backdoor to decrypt the ciphertext. To protect against these threats, assuming the adversary cannot access at least two parts, they suggest to divide the ciphertext into multiple parts and store each part on a different server after applying the AONT. 

\subsection{Our Contributions}
Our goal in this paper is to develop the basic mathematical theory of asymmetric AONTs. In Section \ref{security.sec}, we discuss a combinatorial approach to asymmetric AONTs, and we examine how different combinatorial definitions impact the security of the transforms. We also present some connections with other combinatorial structures such as orthogonal arrays and split orthogonal arrays. Section \ref{linear.sec} focusses on existence and bounds for linear asymmetric AONTs. We complete the solution of the existence problem for $t_i = 1$, as well as when $t_i = 2$ and $t_o = s-1$. Then we turn to cases where $t_i \geq 2$. We  prove a general necessary condition for existence, and then we consider the case $t_i = 2$ in detail. New existence results are obtained from computer searches. Finally, Section \ref{summary.sec} is a brief summary.

We note that many of the results in this paper were first presented in the PhD thesis of the first author
 \cite{Phd}.

\section{Combinatorial Definitions and Security Properties}
\label{security.sec}

Definition \ref{def1} is phrased in terms of security properties, i.e., it specifies information about a subset of inputs that can be deduced if only a certain subset of the outputs is known. (As mentioned in the introduction, we are studying AONTs in the setting of unconditional security.) It is useful to employ a combinatorial description of AONTs in order to analyze them from a mathematical point of view. 
Combinatorial definitions of AONTs have appeared in numerous papers, beginning in \cite{St}. However, the connections between security definitions and combinatorial definitions turn out to be a bit subtle, as was recently shown by Esfahani and Stinson \cite{ES2}.

First, as noted in \cite{ES2}, there are two possible ways to interpret the security requirement. In the original definition of AONT due to Rivest \cite{R}, as well as in Definition \ref{def1}, we only require that the values of any $t_i$ inputs are \emph{completely undetermined}, given the values of $s - t_o$ outputs.
In other words, assuming that every possible input $s$-tuple occurs with positive probability, the probability that the $t_i$ specified inputs take on any specified possible values (given all but $t_o$ outputs) is positive. This notion is termed \emph{weak security} in\cite{ES2}.

An alternative notion that is discussed in detail in \cite{ES2} is that of \emph{perfect security}. Here, we require that the \emph{a posteriori} distribution on any $t_i$ inputs, given the values of $s - t_o$ outputs, is identical to the \emph{a priori} distribution on the same inputs. Thus, \emph{no information} about any $t_i$ inputs is revealed when $s - t_o$ outputs are known. 

The standard combinatorial definition for $(t,t,s,v)$-AONT (see, e.g., \cite{DES,EGS}) involves certain unbiased arrays. We review this definition now and discuss when weak or perfect security can be attained (the security  may depend on the probability distribution defined on the input $s$-tuples). Then we generalize this approach to handle the slightly more complicated case of asymmetric AONTs.

An \emph{$(N,k,v)$-array} is an $N$ by $k$ array, say $A$, whose entries are elements chosen from an alphabet $\Gamma$ of order $v$.  
Suppose the $k$ columns of $A$ are
labelled by the elements in the set $C$.  
Let $D \subseteq C$, and define $A_D$ to be the array obtained from $A$
by deleting all the columns $c \notin D$.
We say that $A$ is \emph{unbiased} with respect to $D$ if the rows of
$A_D$ contain every $|D|$-tuple of elements of $\Gamma$ 
exactly $N / v^{|D|}$ times. Of course, this requires that $N$ is divisible by $v^{|D|}$.

An AONT, say $\phi$, is a bijection from $\Gamma$ to $\Gamma$, where $\Gamma$ is a $v$-set.
The \emph{array representation} of $\phi$ is a  $(v^s,2s,v)$-array, say $A$, that is constructed as follows:
For every input $s$-tuple $(x_1, \dots , x_s) \in \Gamma^s$, there is a row of $A$ containing the entries $x_1, \dots , x_s, y_1, \dots , y_s$, 
where $\phi(x_1, \dots , x_s) = (y_1, \dots , y_s)$. 

Our combinatorial definition of an AONT, Definition \ref{defunbiased}, involves 
arrays that are unbiased with respect to certain subsets of columns. This definition is an obvious generalization of previous definitions for $(t,t,s,v)$-AONTs from \cite{DES,EGS}.

\begin{Definition}
\label{defunbiased}
A \emph{$(t_i,t_o,s,v)$-all-or-nothing transform} is a $(v^s,2s,v)$-array, say $A$, with columns 
labelled $1, \dots , 2s$, 
that is unbiased with respect to the following subsets of columns:
\begin{enumerate}
\item $\{1, \dots , s\}$,
\item $\{s+1, \dots , 2s\}$, and
\item $I \cup J$, 
         for all $I \subseteq \{1,\dots , s\}$ with $|I| = t_i$ and all 
         $J \subseteq \{s+1,\dots , 2s\}$ with $|J| = s-t_o$.
\end{enumerate}
\end{Definition}

We interpret the first $s$ columns of $A$ as indexing the $s$ inputs and the last $s$ columns as indexing the $s$ outputs.
Then, as mentioned above, properties 1 and 2 ensure that the array $A$ defines a bijection $\phi$. Property 3 
guarantees that knowledge of any $s-t_o$ outputs does not rule out any possible values for any $t_i$ inputs.

The following results concerning $(t,t,s,v)$-AONTs are from \cite{ES2}.

\begin{Theorem}
\label{sym.thm}
Suppose $\phi : \Gamma^s \rightarrow \Gamma^s$ is a bijection, where $\Gamma$ is an alphabet of size $v$, and suppose $1 \leq t\leq s$. \vspace{-.2in}\\
	\begin{enumerate}
	\item  Suppose any input $s$-tuple occurs with positive probability. Then the mapping $\phi$ is  a weakly secure AONT if and only if its array representation is a $(t,t,s,v)$-AONT.
	\item The mapping $\phi$ is a perfectly secure AONT if and only if its array representation is a $(t,t,s,v)$-AONT and every input $s$-tuple occurs with the same probability.
	\end{enumerate}
\end{Theorem}

When we turn to asymmetric AONTs, there is an additional subtlety, namely that we can obtain weak security for combinatorial structures that are weaker than the arrays defined  in Definition \ref{defunbiased}.
We can characterize asymmetric AONTs achieving weak security in terms of arrays that satisfy covering properties with respect to certain sets of columns. 
As before, suppose $A$ is 
an $(N,k,v)$-array, whose entries are elements chosen from an alphabet $\Gamma$ of order $v$ and whose  
 columns  are labelled by the the set $C$.  
Also, for $D \subseteq C$,  define $A_D$ as before.
We say that $A$ is
\emph{covering} with respect to a subset of columns $D \subseteq C$ if the rows of
$A_D$ contain every $|D|$-tuple of elements of $\Gamma$ 
\emph{at least once}.

\begin{Remark}
{\rm
An array that satisfies the covering property for all subsets of $t$ columns is called a \emph{$t$-covering array}. Such arrays have many important applications, including software testing. See \cite[\S VI.10]{CD} for a brief survey of covering arrays.
}
\end{Remark}

We state a few simple observations without proof.

\begin{Lemma} 
\label{unbiased-covering.lem}
Suppose $A$ is  an $(N,k,v)$-array with columns labelled by $C$. 
\begin{enumerate}
\item If $A$ is unbiased or covering with respect to $D \subseteq C$, then $N \geq v^{|D|}$.
\item If $A$ is unbiased with respect to $D \subseteq C$, then $A$ is covering with respect to $D$.
\item If $D \subseteq C$ and $N = v^{|D|}$, then $A$ is unbiased with respect to $D$ if and only if $A$ is covering with respect to $D$.
\item If $A$ is unbiased or covering with respect to $D \subseteq C$, then $A$ is unbiased or covering (resp.) with respect to all $D' \subseteq D$.
\end{enumerate}
\end{Lemma}

\begin{Definition}
\label{weakAONT}
A \emph{$(t_i,t_o,s,v)$-weak-all-or-nothing transform} is a $(v^s,2s,v)$-array, say $A$, with columns 
labelled $1, \dots , 2s$, 
that is covering with respect to the following subsets of columns:
\begin{enumerate}
\item $\{1, \dots , s\}$,
\item $\{s+1, \dots , 2s\}$, and
\item $I \cup J$, 
         for all $I \subseteq \{1,\dots , s\}$ with $|I| = t_i$ and all 
         $J \subseteq \{s+1,\dots , 2s\}$ with $|J| = s-t_o$.
\end{enumerate}
\end{Definition}

We note that a $(t,t,s,v)$-weak-AONT is equivalent to a $(t,t,s,v)$-AONT. This follows immediately from Lemma \ref{unbiased-covering.lem}. However, a $(t_i,t_o,s,v)$-weak-AONT is not necessarily a 
$(t_i,t_o,s,v)$-AONT if $t_i < t_0$. Example \ref{weakAONT.exam} depicts a 
$(1,2,3,2)$-weak-AONT that is not a  
$(1,2,3,2)$-AONT.

\begin{Example}
\label{weakAONT.exam}
{\rm
We present a $(1,2,3,2)$-weak AONT over the alphabet $\{a,b\}$. The array representation of this AONT is as follows:
		\begin{center}
			\begin{tabular}{|c|c|c||c|c|c|} \hline
				$x_1$ & $x_2$ & $x_3$ & $y_1$ & $y_2$ & $y_3$ \\ \hline
				a & a & a & a & a & a \\
				a & a & b & b & b & a \\
				a & b & a & b & a & b \\
				a & b & b & b & a & a \\
				b & a & a & a & b & b \\
				b & a & b & a & b & a \\
				b & b & a & a & a & b \\
				b & b & b & b & b & b \\\hline				
			\end{tabular}
		\end{center}

	This array is biased with respect to various pairs of columns $(x_i, y_j)$. For example, we verify that this array is biased with respect to columns $x_1$ and $y_1$. Specifically, the ordered pairs $(a,a)$ and $(b,b)$ each occur once, but the ordered pairs $(a,b)$ and $(b,a)$ each occur three times. 
		
	However, for all choices of $x_i$ and $y_j$, it can be verified that $A$ is covering 
	with respect to the pair of columns $(x_i, y_j)$.} $\blacksquare$
\end{Example}

The following theorem extends part of Theorem \ref{sym.thm} to the asymmetric case.  Proofs are omitted, as they are essentially the same as the proofs in \cite{ES2}.

\begin{Theorem}
\label{asym.thm}
Suppose $\phi : \Gamma^s \rightarrow \Gamma^s$ is a bijection, where $\Gamma$ is an alphabet of size $v$, and suppose $1 \leq t_i \leq t_o \leq s$. \vspace{-.2in}\\
	\begin{enumerate}
	\item  Suppose any input $s$-tuple occurs with positive probability. Then the mapping $\phi$ is  weakly secure  if and only if its array representation is a $(t_i,t_o,s,v)$-weak-AONT.
	\item The mapping $\phi$ is  perfectly secure if its array representation is a $(t_i,t_o,s,v)$-AONT and every input $s$-tuple occurs with the same probability.
	\end{enumerate}
\end{Theorem}

\begin{Remark}
\label{converse.rem}
The second part of Theorem \ref{sym.thm} is ``if and only if''. However, we do not know if the converse of the second part of Theorem \ref{asym.thm} is true when $t_i < t_o$.
\end{Remark}

\subsection{General Properties}

In the rest of the paper, we focus on $(t_i,t_o,s,v)$-AONTs that satisfy Definition \ref{defunbiased}. These are the AONTs that are unbiased with respect to various subsets of columns. First, we record various general properties about these AONTs. Some of these results are generalizations of previous results pertaining to $(t,t,s,v)$-AONT, and most of them follow easily from Lemma \ref{unbiased-covering.lem}.

The following result was shown in \cite{WCJ} for the case $t_i = t_o$. The generalization to arbitrary $
t_i\leq t_o$ is obvious.

\begin{Theorem}
	\label{inverse1.cor}
	A mapping $\phi: \mathcal{X}^s\to \mathcal{X}^s$ is a $(t_i,t_o,s,v)$-AONT if and only if $\phi^{-1}$ is an $(s-t_o,s-t_i,s,v)$-AONT.
\end{Theorem}

\begin{proof} Interchange the first $s$ columns and the last $s$ columns in the array representation of
the AONT $\phi$. 
\end{proof}

An \emph{orthogonal array} \emph{OA$(t,k, v)$}
is a $(v^{t},n,v)$ array, say $A$, that  is unbiased with respect to any $t$ columns.
The next theorem generalizes \cite[Corollary 35]{EGS}.

\begin{Theorem}
	\label{OAthenAONT}
	If there exists an OA$(s,2s,v)$, then there exists a $(t_i,t_0,s,v)$-AONT
	for all $t_i$ and $t_o$ such that $1 \leq t_i \leq t_o \leq s$.
\end{Theorem}

\begin{proof} 
It suffices to show that an OA$(s,2s,v)$ satisfies the conditions of Definition \ref{defunbiased}.
This follows immediately from Lemma \ref{unbiased-covering.lem}  and the observation that
\[ 1 \leq t_i + s - t_o \leq s\]
for all $t_i$ and $t_o$ such that $1 \leq t_i \leq t_o \leq s$.
\end{proof}

Levenshtein \cite{Levenshtein} defined \emph{split orthogonal arrays} (or \emph{SOAs}) as follows. A
split orthogonal array \emph{SOA$(t_1,t_2;n_1,n_2; v)$}
is a $(v^{t_1+t_2},n_1+n_2,v)$ array, say $A$, that satisfies the following properties:
\begin{enumerate}
	\item the columns of $A$ are partitioned into two sets, of sizes $n_1$ and $n_2$, respectively, and
	\item $A$ is unbiased with respect to any $t_1+t_2$ columns in which $t_1$ columns are chosen from the first set of $n_1$ columns and
	$t_2$ columns are chosen from the second set of $n_2$ columns.
\end{enumerate} 
From the definition of split orthogonal arrays, we can immediately obtain the following theorem.

\begin{Theorem}
\label{SOA.thm}
	Suppose there exists a $(t_i, t_o, s, v)$-AONT. Then there exists an SOA$(t_i, s-t_o, s, s,v)$.
\end{Theorem}

\begin{proof}
	Consider the array representation of a $(t_i,t_o,s,q)$-AONT. Denote $n_1 = s, n_2=s, t_1 = t_i$ and $t_2 = s-t_o $. Fixing any  $t_2$ outputs does not yield any information about any $t_1$ inputs. Hence, the array is unbiased with respect to any $s-t_o+t_i$ columns where $t_i$ columns are chosen from the first set of $s$ columns and $s - t_o$ columns are chosen from the second set of $s$ columns. Therefore the array is an SOA$(t_i, s-t_o, s, s,v)$.
	\end{proof}

Theorems \ref{OAthenAONT} and \ref{SOA.thm} show that, in a certain sense, AONTs (symmetric and asymmetric) are ``between'' orthogonal arrays and split orthogonal arrays. More precisely,   
an OA$(s,2s,v)$ implies the existence of a $(t_i, t_o, s, v)$-AONT (for $1 \leq t_i \leq t_o \leq s$), which in turn implies the existence of an SOA$(t_i, s-t_o, s, s,v)$.

\section{Linear Asymmetric AONTs}
\label{linear.sec}

Suppose $q$ is a prime power. If every output  of a $(t_i,t_o,s,v)$-AONT is an $\eff_q$-linear function of the inputs,  the AONT is a \emph{linear} $(t_i,t_o,s,q)$-AONT. Note that we will write a linear 
$(t_i,t_o,s,q)$-AONT in the form $\mathbf{y} =  \mathbf{x} M^{-1}$, where $M$ is an invertible $s$ by $s$ matrix over $\eff_q$ (as always, $\mathbf{x}$ is an input $s$-tuple and $\mathbf{y}$ is an output $s$-tuple). Of course it holds also that $\mathbf{x} = \mathbf{y}M$.

\begin{Remark}
\label{bastion-linear}
The $(1, 2, s, 2)$-AONT described in Example \ref{Ex:even_bastion} (for even values of $s$)  is a linear AONT, where 
$M$ is the $s$ by $s$ matrix with $0$'s on the diagonal and $1$'s elsewhere.
When $s$ is even, $M$ is invertible and $M^{-1} = M$.
\end{Remark}

The following lemma generalizes \cite[Lemma 1]{DES}.

\begin{Lemma}
	\label{linearAsymAONT}
	Suppose that $q$ is a prime power and $M$ is 
	an invertible $s$ by $s$ matrix with entries from $\eff_q$. Suppose $1 \leq t_i \leq t_o \leq s$. 
	Then the function $\mathbf{y} =  \mathbf{x} M^{-1}$ defines a linear $(t_i,t_o,s,q)$-AONT
	if and only if every $t_o$ by $t_i$ submatrix of $M$ has rank $t_i$.
\end{Lemma}

\begin{proof}
  Suppose $I,J \subseteq \{1, \dots , s\}$, $|I| = t_i$, $|J| = t_o$.
	Let ${\mathbf x'} = (x_i : i \in I)$. We have ${\mathbf x'} = {\mathbf y} M'$, where $M'$ is the $s$ by $t_i$  matrix formed from $M$ by deleting all columns not in $I$.
	Now assume that $y_j$ is fixed for all $j \not\in J$ and denote ${\mathbf y'} = (y_j : j \in J)$. Then we can write
	${\mathbf x'} = {\mathbf y'} M'' + {\mathbf c}$, where $M''$ is the $t_o$ by $t_i$ 
submatrix of $M$ formed from $M$ by deleting all columns not in $I$ and all rows not in $J$, 
and ${\mathbf c}$ is a vector of constants.
	If $M''$ is of rank $t_i$, then ${\mathbf x'}$ is completely undetermined, in the sense that
	${\mathbf x'}$ takes on all values in $(\eff_q)^{t_i}$ as ${\mathbf y'}$ 
	varies over $(\eff_q)^{t_o}$. On the other hand, if $t' =\mathsf{rank}(M'')  < t_i$, then
	${\mathbf x'}$ can take on only $(\eff_q)^{t'}$ possible values.
\end{proof}

The following corollaries pertain to the special case where $t_i = t_o = t$. 

\begin{Corollary} 
\cite{DES}
	Suppose $M$ is an  invertible $s$ by $s$ matrix with entries from $\eff_q$. Then $\mathbf{y}  = \mathbf{x} M^{-1}$ 
	defines a linear $(t,t,s,q)$-AONT if and only if every $t$ by $t$ submatrix of $M$ is invertible.
\end{Corollary}

\begin{Corollary}
	\label{inverse2.cor}
	Suppose that $\mathbf{y} = \mathbf{x} M^{-1}$ defines a linear $(t,t,s,q)$-AONT. Then $\mathbf{y}  = \mathbf{x} M$ defines a linear $(s-t,s-t,s,q)$-AONT.
\end{Corollary}

\begin{Corollary}
	\label{linear-st}
	Suppose $M$ is an  invertible $s$ by $s$ matrix with entries from $\eff_q$. Then $\mathbf{y}  = \mathbf{x} M^{-1}$ 
	defines a linear $(t,t,s,q)$-AONT if and only if every $s-t$ by $s-t$ submatrix of $M^{-1}$ is invertible.
\end{Corollary}

Another approach to construct asymmetric AONTs is to use $t$-AONTs or other asymmetric AONTs. The following results will present various such constructions. First, we generalize  \cite[Theorem 20]{EGS}.

\begin{Lemma}
	\label{Lem:AsymCofactor}
	If $1 \leq t_i \le t_o < s$, then the existence of a linear $(t_i,t_o,s,q)$-AONT implies the existence of a linear $(t_i,t_o,s-1,q)$-AONT.
\end{Lemma}
\begin{proof}
	Let $M$ be a matrix for a linear $(t_i,t_o,s,q)$-AONT. Since $M$ is invertible, if we calculate its determinant using the cofactor expansion of $M$ with respect to its first row, at least one of the $(s-1)\times (s-1)$ submatrices is invertible. Also, any $t_o\times t_i$ submatrix of $M$, including those in the invertible submatrix, are of rank $t_i$. Hence, the invertible submatrix is a $(t_i,t_o,s-1,q)$-AONT.
\end{proof}

\begin{Lemma}\label{Lem:Asym_L_to}
	If  $1 \leq t_i\le t_o \le s$, then the existence of a linear $(t_i,t_o,s,q)$-AONT implies the existence of a linear $(t_i,t^{\prime}_o,s,q)$-AONT for all $t^{\prime}_o$ such that  $t_o \leq t^{\prime}_o \leq  s$.
\end{Lemma}

\begin{proof}
	Consider the matrix representation of the linear $(t_i,t_o,s,q)-$AONT. Every $ t^{\prime}_o$ by $t_i$ submatrix is rank $t_i$, because all its $t_o \times t_i$ submatrices are of rank $t_i$.
\end{proof}

\begin{Lemma}\label{Lem:Asym_S_ti}
	If  $1 \leq t_i\le t_o \le s$, then the existence of a linear $(t_i,t_o,s,q)$-AONT implies the existence of a linear $(t^{\prime}_i,t_o,s,q)$-AONT for any  $t^{\prime}_i$ such that  $1 \leq t^{\prime}_i \leq t_i \leq  s$.
\end{Lemma}

\begin{Example}
\label{E232}
We observe that  existence of a linear $(t_i,t_o,s,q)$-AONT does not necessarily imply the existence of a linear $(t_i,t_i,s,q)$-AONT or a linear $(t_o,t_o,s,q)$-AONT. Consider the linear $(2,3,4,2)$-AONT presented by the following matrix
	\[
	\left(\begin{array}{c c c c}
	1 & 0 & 0 & 1 \\
	0 & 1 & 0 & 1 \\
	0 & 0 & 1 & 1 \\
	1 & 1 & 1 & 0
	\end{array} \right).
	\]
	While every $3\times 2$ submatrix of the matrix above is of rank $2$, a $(2,2,s,2)$-AONT does not exist if $s>2$, as was proven by D'Arco et al.\ \cite{DES}. 
	
Additionally,  from Corollary \ref{inverse2.cor}, a linear 	$(3,3,4,2)$-AONT would be equivalent to a linear $(1,1,4,2)$-AONT. Since it was shown in  \cite{St} that an $(1,1,4,2)$-AONT does not exist, we conclude that a	linear 	$(3,3,4,2)$-AONT does not exist. $\blacksquare$
\end{Example}

The main general construction for linear $(t,t,s,q)$-AONTs in \cite{DES} uses Cauchy matrices. We  provide a generalization that applies to asymmetric AONTs.

\begin{Theorem}
Suppose $q \geq 2s$ is a prime power and $1 \leq t_i \leq t_o \leq s$. Then there exists a linear $(t_i,t_o,s,q)$-AONT.
\end{Theorem}

\begin{proof} 
In \cite[Theorem 2]{DES}, it was shown that a linear $(t,t,s,q)$-AONT exists if $q \geq 2s$ is a prime power and $1 \leq t \leq s$. Let $t_i = t_o = t$ and then apply Lemma \ref{Lem:Asym_S_ti}. This shows that there
is a a linear $(t'_i,t_o,s,q)$-AONT provided that  $1 \leq t'_i \leq t_o \leq s$.

\end{proof}

\subsection{Linear $(1,t_o,s,q)$-AONT}

We  noted in Remark \ref{bastion-linear} that there exists a linear $(1, 2, s, 2)$-AONT for all even values of $s \geq 2$. 
In the next lemma, we show that   linear $(1,2,s,2)$-AONTs exist for odd values of $s$.

\begin{Lemma}
\label{OddBastion}
There is a  linear $(1,2,s,2)$-AONT for any odd value of $s \geq 3$.
\end{Lemma}

\begin{proof}
Suppose $s\geq 3$ is odd. Let $M$ be the $s$ by $s$ matrix whose first subdiagonal consists of $0$'s, but all other entries are $1$'s. For example, when $s=5$, we have	
		\[
	M = \left(\begin{array}{c c c c c}
	1 & 1 & 1 & 1 & 1\\
	0 & 1 & 1 & 1 & 1\\
	1 & 0 & 1 & 1 & 1\\
	1 & 1 & 0 & 1 & 1\\
	1 & 1 & 1 & 0 & 1
	\end{array}\right).\]
The matrix $M$ is invertible and its inverse is an $s$ by $s$ matrix with a right top submatrix that is an identity matrix, and $1$'s occur along the last row and first column. For example, when $s=5$, we have
	\[
	M^{-1} = \left(\begin{array}{c c c c c}
	1 & 1 & 0 & 0 & 0 \\
	1 & 0 & 1 & 0 & 0 \\
	1 & 0 & 0 & 1 & 0 \\
	1 & 0 & 0 & 0 & 1 \\
	1 & 1 & 1 & 1 & 1
	\end{array}\right).\]
Further, any $2$ by $1$ submatrix of $M$ has rank $1$ because there is at most one occurrence of $0$ in each column of $M$. 
\end{proof}

Recall that $\mathbf{y} =  \mathbf{x} M^{-1}$ and  $\mathbf{x} = \mathbf{y}M$.	
Given $s$ inputs $x_1, \dots , x_s \in \zed_2$, the above-discussed transform can be computed as follows:
\[
\begin{array}{l}
y_1 = \displaystyle\sum_{i = 1}^s x_i.\Botstrut{4}\\
y_i = x_{i-1} + x_{s}, \text{ for } 2\le i \le s.
\end{array}
\]
 This yields the $s$ outputs $y_1, \dots , y_s$. The inverse transform  is computed as
	\[
\begin{array}{l}
x_s = \displaystyle\sum_{i = 1}^s y_i \Botstrut{4} \\ 
x_i = x_s + y_{i+1}, \text{ for } 2\le i \le s. 
\end{array}
\]
Thus, computation of the transform or its inverse requires $2s-2$ addition operations in $\zed_2$
(i.e., exclusive-ors). 

\begin{Theorem}
Suppose $q$ is a prime power and $1 \leq t_o \leq s$. 
Then there is a linear $(1,t_o,s,q)$-AONT unless $q=2$ and $t_0 = 1$.  Further, there does not exist
any $(1,1,s,2)$-AONT.
\end{Theorem}

\begin{proof}
When $q > 2$, it was shown in \cite[Corollary 2.3]{St} that there exists a linear $(1,1,s,q)$-AONT for all $s \geq 1$. Applying  Lemma \ref{Lem:Asym_L_to}, there exists a linear $(1,t_o,s,q)$-AONT for all prime powers $q > 2$ and all $t_0$ and $s$ such that $1 \leq t_o \leq  s$.

We have also noted in Remark \ref{bastion-linear} that there exists a linear $(1, 2, s, 2)$-AONT for all even values of $s \geq 2$. Applying  Lemma \ref{Lem:Asym_L_to}, there exists a linear $(1,t_o,s,2)$-AONT for all $t_0$ and $s$ such that $s$ is even and $2 \leq t_o \leq  s$. From Lemmas \ref{Lem:Asym_L_to} and \ref{OddBastion}, there exists a linear $(1,t_o,s,2)$-AONT for all $t_0$ and $s$ such that $s$ is odd and $2 \leq t_o \leq  s$.

Finally, it was shown in \cite{St} that there does not exist
any $(1,1,s,2)$-AONT.
\end{proof}

\subsection{Linear $(2,s-1,s,2)$-AONT}

In this section, we consider linear $(2,s-1,s,2)$-AONTs. 

For even values of $s \geq 4$, we use the $(1,2,s,2)$-AONT from Remark \ref{bastion-linear}.
This AONT  is based on the  $s$ by $s$ matrix $M$ with $0$'s on the diagonal and $1$'s elsewhere. We have already noted that this matrix is invertible. To show that it gives rise to a $(2,s-1,s,2)$-AONT, we need to show that any $s-1$ by $2$ submatrix has rank $2$.  It can be observed that any choice of $s-1$ rows and two columns will contain at least $s-3 \geq 1$ occurrences of the row $(1,1)$ and at least one copy of the row $(0,1)$ or $(1,0)$. Therefore, we have proven the following.

\begin{Lemma}
	\label{s-even.lem}
	For any even integer $s\ge 4$, there exists a linear $(2,s-1,s,2)$-AONT.
\end{Lemma}


Now we turn to odd values of $s$.

\begin{Lemma}
	\label{s-odd.lem}
	For any odd integer $s\ge 5$, there is a linear $(2,s-1,s,2)$-AONT.
\end{Lemma}

\begin{proof}
For an odd integer $s \geq 5$, define the $s$ by $s$ matrix $B_s$ to have $1$'s in the entries on the main diagonal, the last row and the last column, and $0$'s elsewhere. 

For example, the matrix $B_5$ is as follows:
\[ \left(\begin{array}{c c c c c }
1 & 0 & 0 & 0 & 1 \\
0 & 1 & 0 & 0 & 1 \\
0 & 0 & 1 & 0 & 1 \\
0 & 0 & 0 & 1 & 1 \\
1 & 1 & 1 & 1 & 1 \end{array}
\right).
\]

	Suppose we subtract rows $1, \dots , s-1$ of $B_s$ from row $s$.
	Then we obtain an upper triangular matrix with $1$'s on the main diagonal.  This  proves that $B_s$ is invertible.
	
	Now we prove that any $s-1$ by $2$ submatrix has rank two. First, consider columns $i$ and $s$, where $1 \leq i \leq s-1$. The $s$ rows of this submatrix contain two copies of $(1,1)$ and $s-2$ copies of $(0,1)$. Therefore, any $s-1$ rows still contain at least one copy of $(1,1)$ and at least one copy of $(0,1)$.
	This means that the $s-1$ by $2$ submatrix has rank $2$.
	
Next, we consider columns $i$ and $j$, where $1 \leq i < j \leq s-1$. 	The $s$ rows of this submatrix contain one copy of each of $(0,1)$, $(1,0)$ and $(1,1)$. Therefore, any $s-1$ rows still contain 
at least two of the three pairs $(0,1)$, $(1,0)$ and $(1,1)$.
	This means that the $s-1$ by $2$ submatrix has rank $2$.
	\end{proof}

\begin{Theorem}
	There is a linear  $(2,s-1,s,2)$-AONT if and only if $s\geq 4$.
\end{Theorem}

\begin{proof}
If a $(t_i,t_o,s,q)$-AONT exsits, we must have $t_i \leq t_o$. 
Hence, $s \geq 3$ if a $(2,s-1,s,2)$-AONT exists.
D'Arco et al.\ \cite{DES} proved that a linear $(2,2,3,2)$-AONT does not exist.  
For $s \geq 4$, Lemmas \ref{s-even.lem} and \ref{s-odd.lem} show that a linear $(2,s-1,s,2)$-AONT exists.
\end{proof}

\subsection{Linear $(t_i,t_o,s,q)$-AONT with $t_i \geq 2$}

In this section, we study linear $(t_i,t_o,s,q)$-AONTs with $t_i \geq 2$. We first prove a general upper bound on $s$ as a function of $t_i$, $t_o$ and $q$. Then we consider the case $t_i=2$ in detail.


\begin{Theorem}
\label{T1}
Suppose there exists a linear $(t_i,t_o,s,q)$-AONT with $2\leq t_i \leq t_o$. Then the following bound holds:
	\[ s \leq \frac{(t_o -1)(q^{t_i} - 1)}{(t_i -1)(q-1)}.\]
\end{Theorem}

\begin{proof}
Fix any $t_i$ columns of the matrix $M$ and consider the resulting submatrix $M'$. Recall that any $t_o$ by $t_i$ submatrix of $M$ must have rank $t_i$. 

There are $q^{t_i}$ possible $t_i$-tuples for any given row of $M'$. We can replace an all-zero $t_i$-tuple with any other $t_i$-tuple, and it does not decrease the rank of any $t_o$ by $t_i$ submatrix in $M'$. Hence, we can assume that there is no all-zero $t_i$-tuple among the rows of $M'$. Therefore, there are $q^{t_i}-1$ possible rows in $M'$. 

For any two nonzero $t_i$-tuples, say $a$ and $b$, define $a\sim b$ if there is a nonzero element $\alpha \in \eff_{q}$ such that $a = \alpha b$. Clearly $\sim$ is an equivalence relation, and there are $({q^{t_i}-1})/({q-1})$ equivalence classes, each of size $q-1$.

	Suppose the equivalence classes of rows are denoted by $\mathcal{E}_i$. Further, suppose there are $a_i$ rows from $\mathcal{E}_i$ in $M'$, for $1 \leq i \leq (q^{t_i}-1)/(q-1)$. The sum of the $a_i$'s is equal to $s$ and hence the average value of an $a_i$ is 
	\[ \overline{a} = \frac{s(q-1)}{  q^{t_i}-1}.\] 
	Let $L$ denote the sum of the $t_i -1$ largest $a_i$'s. It is clear that
	\[ L \geq (t_i - 1)\overline{a} = \frac{s(t_i - 1)(q-1)}{  q^{t_i}-1}.\]
	Also, because the $t_o$ rows of $M'$ cannot come from fewer than $t_i$ equivalence classes,
we have
	\[ L \leq  t_o - 1.\] Hence, combining the two inequalities, we see that 
	\[ s \leq \frac{(t_o -1)(q^{t_i} - 1)}{(t_i -1)(q-1)}.\]
\end{proof}

We now look at the case $t_i = 2$ in more detail. 

\begin{Theorem}
	\label{Thrm:2_to_AsymAONT_bound}
	Suppose there exists a linear $(2,t_o,s,q)$-AONT with $2 \leq t_o$. Then the following bound holds:
	\[s \leq  \max \{ 1 + (t_o - 2)(q+1), 2 + (t_o - 1)(q-1)\}.\]
\end{Theorem}

\begin{proof}
	Consider an $s$ by $2$ submatrix $M'$ and let $a_0$ be the number of $(0,0)$ rows in this submatrix.
	We divide the proof into two cases. 
	
	\begin{description}
		
		\item[case (1)] \mbox{\quad} \vspace{.1in}\\
		Suppose $a_0 \geq 1$. We claim that $M'$ contains at most $t_o - a_0 - 1$ rows from any one equivalence class $\mathcal{E}_i$, where equivalence classes are as defined in the proof of Theorem \ref{T1}.
		This follows because $t_o - a_0$ rows from one equivalence class, together with the $a_0$ rows of $0$'s, would result in $M'$ having rank 1. 
		Excluding the rows of $0$'s,  there are $q+1$ possible equivalence classes of rows, 
		so
		\[ s \leq a_0  + (t_o - a_0 - 1 )(q+1) \leq 1 + (t_o - 2)(q+1).\]
		
		\item[case (2)] \mbox{\quad} \vspace{.1in}\\
		If we are not in case (1), then  $a_0 = 0$ for \emph{every} $s$ by $2$  submatrix $M'$. There can be at most one $0$ in each row of $M$, so there are at most $s$ occurrences of $0$ in $M$. Therefore, there must be two columns in $M$ that contain a total of at most two $0$'s. We focus on this particular $s$ by $2$ submatrix $M'$. 
		
		Let the number of $0$'s in $M'$ be denoted by $a$; we have noted that $a  \leq 2$. 
		In the  $s - a$ rows that do not contain a $0$, there are at most $t_o - 1$ rows from any equivalence class $\mathcal{E}_i$. Note that we have excluded two $\mathcal{E}_i$'s, i.e., $(*,0)$ and $(0,*)$,  so  
		\[ s \leq a  + (t_o  - 1 )(q- 1) \leq 2 + (t_o - 1)(q-1).\]
	\end{description}
	
	Since one of the above two cases must hold, we have
	\[s \leq  \max \{ 1 + (t_o - 2)(q+1), 2 + (t_o - 1)(q-1)\}.\]
\end{proof}

We note that
\[ 1 + (t_o - 2)(q+1) < (t_o -1)(q+1)\]
and 
\[ 2 + (t_o - 1)(q-1) < (t_o - 1)(q+1),\]
so 
\[ \max \{ 1 + (t_o - 2)(q+1), 2 + (t_o - 1)(q-1)\} < (t_o - 1)(q+1).\]
Hence the bound from Theorem \ref{Thrm:2_to_AsymAONT_bound} improves  Theorem \ref{T1} when $t_i = 2$.

	For positive integers $t_i$ and $t_o$, where $1 \leq t_i \leq t_o$, and a prime power $q$, define \[S(t_i,t_o,q)= \max \{s: \text{a linear } (t_i,t_o,s,q)\text{-AONT exists}\}.\] 
Note that $S(t_i,t_o,q) \geq t_o$ because the $t_o$ by $t_o$ identity matrix is a $(t_i,t_o,t_o,q)$-AONT.

\begin{Theorem}
 Suppose $1 \leq t_i \leq t_o$ and $q$ is a prime power. Then there exists a $(t_i,t_o,s,q)$-AONT for
 $t_o \leq s \leq S(t_i,t_o,q)$.
\end{Theorem}

\begin{proof}
This is an immediate consequence of Lemma \ref{Lem:AsymCofactor}.
\end{proof}

We mainly consider cases where $2 \leq t_i < t_o$. However, before proceeding, we recall some previous results concerning the special case $t_i = t_o = 2$. Theorems \ref{T1} and \ref{Thrm:2_to_AsymAONT_bound} both assert that $S(2,2,q) \leq q+1$. 
However, the stronger result that $S(2,2,q) \leq q$ was previously shown in \cite[Theorem 14]{EGS}.
There are also some known lower bounds on $S(2,2,q)$, which are recorded in the following theorem.

\begin{Theorem}
Suppose $q$ is a prime power. Then the following bounds hold.
\begin{enumerate}
\item $\lfloor q/2 \rfloor \leq S(2,2,q) \leq q$.
\item $q-1 \leq S(2,2,q) \leq q$ if $q = 2^n -1$ is prime, for some integer $n$.
\item $S(2,2,q) = q$ if $q$ is prime. 
\end{enumerate}
\end{Theorem}

\begin{proof}
1.\ and 2.\ are shown in \cite{EGS}, while 3.\ is proven in \cite{WCJ}.
\end{proof}

The cases when $t_i < t_o$ have not received previous study in the literature.
Theorems \ref{T1} and \ref{Thrm:2_to_AsymAONT_bound} provide upper bounds on $S(t_i,t_o,q)$. 
We evaluate some of these upper bounds for specific families of parameters in Table \ref{tab:ASYMAONT_bounds}.

\begin{table}
	\caption{Examples of bounds from Theorems \ref{T1} and \ref{Thrm:2_to_AsymAONT_bound}.}
	\label{tab:ASYMAONT_bounds}
	\[
	\begin{array}{c | c| c| c | c} 
		t_i & q & t_o  & \text{Upper bound on }S(t_i,t_o,q) & \text{Justification}\\ \hline \hline
		2 & 2 & 2 & t_o+1  & \text{Theorem \ref{Thrm:2_to_AsymAONT_bound}}\\
		2 & 2 & \geq 3 & 3t_o-5 & \text{Theorem \ref{Thrm:2_to_AsymAONT_bound}}\\		\hline
		2 & 3 & 2,3 & 2t_o & \text{Theorem \ref{Thrm:2_to_AsymAONT_bound}} \\
		2 & 3 & \geq 4 & 4t_o-7 & \text{Theorem \ref{Thrm:2_to_AsymAONT_bound}} \\\hline
		2 & 4 & 2,3 & 3t_o-1 & \text{Theorem \ref{Thrm:2_to_AsymAONT_bound}} \\
				2 & 4 & \geq 4 & 5t_o-9 & \text{Theorem \ref{Thrm:2_to_AsymAONT_bound}} \\\hline
		3 & 3 & \text{any} & \Topstrut{3}\Botstrut{1.5}\frac{13(t_o-1)}{2}  & \text{Theorem \ref{T1}}\\\hline
		3 & 4 & \text{any} & 20(t_o-1)  & \text{Theorem \ref{T1}}\\\hline
		3 & 5 & \text{any} & \Topstrut{3}\frac{121(t_o-1)}{2} & \text{Theorem \ref{T1}} 
	\end{array}
	\]
\end{table}

We can also obtain  lower bounds on $S(2,t_o,q)$, for specific choices of $t_o$ and $q$, from computer searches. 
The results of our searches are presented in Examples \ref{A242} to \ref{A237}.  Table \ref{tab:ASYMAONT_CompRes} lists upper and lower bounds on $S(2,t_o,q)$, for some fixed values of $t_o,$ and $q$.  
There are four cases where we can report exact values of $S(2,t_o,q)$. When $(t_0,q) = (3,2)$ and $(3,3)$, we have found examples that meet the upper bounds from  Theorem \ref{Thrm:2_to_AsymAONT_bound}. For $(t_0,q) = (4,2)$ and $(5,2)$, the searches were run to completion and the exact values of $S(2,t_o,q)$ turn out to be strictly less than the bounds obtained from Theorem \ref{Thrm:2_to_AsymAONT_bound}, which are $S(2,4,2) \leq 7$ and $S(2,5,2) \leq 10$.

\begin{table}
	\caption{ Upper and lower  bounds on $S(2,t_o,q)$}
	\label{tab:ASYMAONT_CompRes}
	\centering
	\begin{tabular}{ c| c|  c | l  || c| l} 
		  $t_o$& $q$ & lower bound & reference & upper bound & reference \\ \hline\hline
		   3 & 2 & 4  & Example \ref{E232} & 4 & Theorem \ref{Thrm:2_to_AsymAONT_bound} \\
		   4 & 2 & 5  & Example \ref{A242} & 5 & exhaustive search \\
		   5 & 2 & 8  & Example \ref{A252} & 8 & exhaustive search\\
		   6 & 2 & $10$  & Example \ref{A262} &13& Theorem \ref{Thrm:2_to_AsymAONT_bound}\\
		   7 & 2 & $12$ & Example \ref{A272} & 16& Theorem \ref{Thrm:2_to_AsymAONT_bound}\\
		   8 & 2 & $13$ & Example \ref{A282} & 19 & Theorem \ref{Thrm:2_to_AsymAONT_bound}\\ \hline
		   3 & 3 & 6  & Example \ref{A233}& 6& Theorem \ref{Thrm:2_to_AsymAONT_bound}\\
		   4 & 3 & $8$ & Example \ref{A243}& 9& Theorem \ref{Thrm:2_to_AsymAONT_bound}\\
		   5 & 3 & $9$  & Example \ref{A253}& 13& Theorem \ref{Thrm:2_to_AsymAONT_bound}\\
		   6 & 3 & $13$ & Example \ref{A263}& 17& Theorem \ref{Thrm:2_to_AsymAONT_bound}\\ \hline
		   3 & 4 & $6$ & Example \ref{A234}& 8& Theorem \ref{Thrm:2_to_AsymAONT_bound}\\
		   4 & 4 & $9 $  & Example \ref{A244}& 11& Theorem \ref{Thrm:2_to_AsymAONT_bound}\\
		   5 & 4 & $11 $ & Example \ref{A254}& 16& Theorem \ref{Thrm:2_to_AsymAONT_bound}\\ \hline
		   3 & 5 & $8 $ & Example \ref{A235}& 10& Theorem \ref{Thrm:2_to_AsymAONT_bound}\\
		   4 & 5 & $10 $ & Example \ref{A245}& 14& Theorem \ref{Thrm:2_to_AsymAONT_bound}\\ \hline   
		   3 & 7 & $8 $ & Example \ref{A237}& 14& Theorem \ref{Thrm:2_to_AsymAONT_bound} \\  	  	 
	\end{tabular}
\end{table}

\section{Discussion}
\label{summary.sec}
 
There are many open problems involving asymmetric AONTs. It would certainly be of interest
to find improved necessary conditions and general constructions. The first cases are when $t_i = 2$. A starting point would be to close the gaps in the bounds reported in Table \ref{tab:ASYMAONT_CompRes}.

As mentioned in Remark \ref{converse.rem}, it is unknown if the converse of part 2 of Theorem \ref{asym.thm} is true when $t_i < t_o$. We feel that this question is worthy of further study. 

\section*{Acknowledgements}

We thank Bill Martin for helpful discussions.


\appendix
\section{Appendix}


\begin{Example}
	\label{A242}
	A linear $(2,4,5,2)$-AONT:
	\[
	\left(
	\begin{array}{ccccc}
		1  & 1  & 1  & 0  & 1 \\ 
		1  & 0  & 0  & 1  & 0 \\ 
		0  & 1  & 0  & 1  & 0 \\ 
		0  & 0  & 1  & 1  & 0 \\ 
		0  & 0  & 0  & 0  & 1 
	\end{array}
	\right).
	\]
\end{Example}

\begin{Example}
		\label{A252}
		A linear $(2,5,8,2)$-AONT:
	\[
	\left(
	\begin{array}{cccccccc}
	1  & 1  & 1  & 0  & 0  & 0  & 1  & 1 \\ 
	1  & 1  & 0  & 1  & 0  & 1  & 1  & 0 \\ 
	1  & 0  & 1  & 0  & 1  & 1  & 0  & 0 \\ 
	1  & 0  & 0  & 1  & 1  & 0  & 1  & 0 \\ 
	0  & 1  & 1  & 0  & 1  & 0  & 1  & 0 \\ 
	0  & 1  & 0  & 1  & 1  & 0  & 0  & 1 \\ 
	0  & 0  & 1  & 1  & 0  & 1  & 1  & 1 \\ 
	0  & 0  & 0  & 0  & 1  & 1  & 1  & 1	
	\end{array}
	\right) .
	\]
\end{Example}

\begin{Example}
		\label{A262}
	A linear $(2,6,10,2)$-AONT:
	\[
	\left(
	\begin{array}{cccccccccc}
		1  & 1  & 0  & 1  & 0  & 0  & 0  & 1  & 1  & 1 \\ 
		1  & 1  & 0  & 0  & 1  & 1  & 0  & 0  & 1  & 0 \\ 
		1  & 0  & 1  & 1  & 0  & 1  & 0  & 1  & 0  & 0 \\ 
		1  & 0  & 1  & 0  & 1  & 0  & 1  & 1  & 1  & 0 \\ 
		1  & 0  & 0  & 1  & 1  & 1  & 1  & 0  & 0  & 1 \\ 
		0  & 1  & 1  & 1  & 0  & 1  & 1  & 0  & 1  & 0 \\ 
		0  & 1  & 1  & 0  & 1  & 1  & 0  & 1  & 0  & 1 \\ 
		0  & 1  & 0  & 1  & 1  & 0  & 1  & 1  & 0  & 0 \\ 
		0  & 0  & 1  & 1  & 1  & 0  & 0  & 0  & 1  & 1 \\ 
		0  & 0  & 0  & 0  & 0  & 1  & 1  & 1  & 1  & 1
	\end{array}
	\right) .
	\]\end{Example}

\begin{Example}
		\label{A272}
	A linear $(2,7,12,2)$-AONT:
	\[
	\left(
	\begin{array}{cccccccccccc}
	1  & 1  & 1  & 0  & 0  & 0  & 1  & 0  & 0  & 1  & 0  & 1 \\ 
	1  & 1  & 0  & 1  & 0  & 0  & 0  & 1  & 1  & 1  & 1  & 0 \\ 
	1  & 1  & 0  & 0  & 1  & 1  & 0  & 0  & 1  & 0  & 0  & 1 \\ 
	1  & 0  & 1  & 1  & 0  & 1  & 0  & 1  & 0  & 0  & 0  & 1 \\ 
	1  & 0  & 1  & 0  & 1  & 0  & 1  & 1  & 1  & 0  & 1  & 0 \\ 
	1  & 0  & 0  & 1  & 1  & 1  & 1  & 0  & 0  & 1  & 1  & 0 \\ 
	0  & 1  & 1  & 1  & 0  & 1  & 1  & 0  & 1  & 0  & 1  & 0 \\ 
	0  & 1  & 1  & 0  & 1  & 1  & 0  & 1  & 0  & 1  & 1  & 0 \\ 
	0  & 1  & 0  & 1  & 1  & 0  & 1  & 1  & 0  & 0  & 0  & 1 \\ 
	0  & 0  & 1  & 1  & 1  & 0  & 0  & 0  & 1  & 1  & 0  & 1 \\ 
	0  & 0  & 0  & 0  & 0  & 1  & 1  & 1  & 1  & 1  & 0  & 1 \\ 
	0  & 0  & 0  & 0  & 0  & 0  & 0  & 0  & 0  & 0  & 1  & 1
	\end{array}
	\right) .
	\]
\end{Example}

\begin{Example}
		\label{A282}
	A linear $(2,8,13,2)$-AONT:
	\[
	\left(
	\begin{array}{ccccccccccccc}
		1  & 1  & 1  & 0  & 0  & 0  & 1  & 0  & 0  & 1  & 0  & 1  & 0 \\ 
		1  & 1  & 0  & 1  & 0  & 0  & 0  & 1  & 1  & 1  & 1  & 0  & 0 \\ 
		1  & 1  & 0  & 0  & 1  & 1  & 0  & 0  & 1  & 0  & 0  & 1  & 1 \\ 
		1  & 0  & 1  & 1  & 0  & 1  & 0  & 1  & 0  & 0  & 0  & 1  & 1 \\ 
		1  & 0  & 1  & 0  & 1  & 0  & 1  & 1  & 1  & 0  & 1  & 0  & 1 \\ 
		1  & 0  & 0  & 1  & 1  & 1  & 1  & 0  & 0  & 1  & 1  & 0  & 0 \\ 
		0  & 1  & 1  & 1  & 0  & 1  & 1  & 0  & 1  & 0  & 1  & 0  & 1 \\ 
		0  & 1  & 1  & 0  & 1  & 1  & 0  & 1  & 0  & 1  & 1  & 0  & 0 \\ 
		0  & 1  & 0  & 1  & 1  & 0  & 1  & 1  & 0  & 0  & 0  & 1  & 1 \\ 
		0  & 0  & 1  & 1  & 1  & 0  & 0  & 0  & 1  & 1  & 0  & 1  & 0 \\ 
		0  & 0  & 0  & 0  & 0  & 1  & 1  & 1  & 1  & 1  & 0  & 1  & 0 \\ 
		0  & 0  & 0  & 0  & 0  & 0  & 0  & 0  & 0  & 0  & 1  & 1  & 0 \\ 
		0  & 0  & 0  & 0  & 0  & 0  & 0  & 0  & 0  & 0  & 0  & 0  & 1
	\end{array}
	\right) .
	\]
\end{Example}

\begin{Example}
		\label{A233}
	A linear $(2,3,6,3)$-AONT:
	\[
	\left(
	\begin{array}{cccccc}
		0  & 1  & 1  & 1  & 1  & 1 \\ 
		1  & 0  & 1  & 1  & 2  & 2 \\ 
		1  & 1  & 0  & 2  & 1  & 2 \\ 
		1  & 1  & 2  & 0  & 2  & 1 \\ 
		1  & 2  & 1  & 2  & 0  & 1 \\ 
		1  & 2  & 2  & 1  & 1  & 0
	\end{array}
	\right) .
	\]\end{Example}

\begin{Example}
		\label{A243}
	A linear $(2,4,8,3)$-AONT:
	\[
	\left(
	\begin{array}{cccccccc}
		0  & 0  & 0  & 1  & 1  & 1  & 1  & 2 \\ 
		0  & 1  & 1  & 0  & 1  & 2  & 2  & 0 \\ 
		0  & 1  & 1  & 1  & 0  & 0  & 1  & 1 \\ 
		1  & 0  & 1  & 0  & 1  & 0  & 1  & 2 \\ 
		1  & 0  & 1  & 1  & 0  & 1  & 2  & 0 \\ 
		1  & 1  & 0  & 2  & 2  & 0  & 1  & 0 \\ 
		1  & 1  & 2  & 0  & 1  & 1  & 2  & 1 \\ 
		1  & 2  & 0  & 1  & 0  & 2  & 1  & 1
	\end{array}
	\right) .
	\]
\end{Example}

\begin{Example}
		\label{A253}
	A linear $(2,5,9,3)$-AONT:
	\[
	\left(
	\begin{array}{ccccccccc}
		0  & 0  & 0  & 0  & 0  & 0  & 0  & 1  & 1 \\ 
		0  & 0  & 0  & 1  & 1  & 1  & 1  & 1  & 1 \\ 
		0  & 1  & 1  & 0  & 0  & 1  & 1  & 1  & 2 \\ 
		0  & 1  & 1  & 1  & 2  & 0  & 2  & 2  & 0 \\ 
		1  & 0  & 1  & 0  & 1  & 0  & 2  & 2  & 1 \\ 
		1  & 0  & 1  & 1  & 0  & 2  & 1  & 2  & 0 \\ 
		1  & 1  & 0  & 2  & 0  & 0  & 1  & 2  & 1 \\ 
		1  & 1  & 2  & 0  & 1  & 1  & 2  & 1  & 0 \\ 
		1  & 2  & 0  & 1  & 2  & 1  & 1  & 2  & 0
	\end{array}
	\right) .
	\]
\end{Example}

\begin{Example}
		\label{A263}
		A linear $(2,6,13,3)$-AONT:
    \[
    \left(
    \begin{array}{ccccccccccccc}
        0  & 0  & 0  & 0  & 0  & 0  & 0  & 1  & 1  & 1  & 1  & 1  & 2 \\ 
        0  & 0  & 0  & 1  & 1  & 1  & 1  & 0  & 1  & 1  & 1  & 2  & 2 \\ 
        0  & 1  & 1  & 0  & 0  & 1  & 1  & 0  & 1  & 1  & 2  & 2  & 0 \\ 
        0  & 1  & 1  & 1  & 1  & 0  & 0  & 1  & 1  & 1  & 0  & 1  & 1 \\ 
        0  & 1  & 1  & 1  & 2  & 1  & 2  & 1  & 0  & 0  & 1  & 2  & 2 \\ 
        1  & 0  & 1  & 0  & 2  & 1  & 1  & 2  & 0  & 2  & 0  & 1  & 2 \\ 
        1  & 0  & 1  & 1  & 0  & 0  & 2  & 0  & 1  & 2  & 2  & 1  & 2 \\ 
        1  & 0  & 1  & 1  & 1  & 2  & 0  & 2  & 2  & 0  & 1  & 2  & 0 \\ 
        1  & 1  & 0  & 2  & 0  & 2  & 0  & 1  & 1  & 2  & 2  & 2  & 2 \\ 
        1  & 1  & 2  & 0  & 1  & 1  & 0  & 0  & 2  & 1  & 2  & 1  & 2 \\ 
        1  & 1  & 2  & 0  & 2  & 0  & 2  & 2  & 1  & 0  & 1  & 1  & 1 \\ 
        1  & 2  & 0  & 1  & 2  & 0  & 1  & 1  & 0  & 1  & 2  & 2  & 0 \\ 
        1  & 2  & 0  & 2  & 1  & 1  & 2  & 0  & 1  & 0  & 0  & 2  & 1
    \end{array}
    \right) .
    \]

\end{Example}

\begin{Example}
		\label{A234}
	A linear $(2,3,6,4)$-AONT:
	\[
	\left(
	\begin{array}{cccccc}
		0  & 0  & 1  & 1  & 1  & 1 \\ 
		0  & 1  & 0  & 1  & 1  & 2 \\ 
		1  & 0  & 0  & 1  & 2  & 1 \\ 
		1  & 1  & 1  & 0  & 1  & 3 \\ 
		1  & 2  & 3  & 0  & 1  & 2 \\ 
		1  & 3  & 2  & 1  & 0  & 2
	\end{array}
	\right) .
	\]

\end{Example}

\begin{Example}
		\label{A244}
	A linear $(2,4,9,4)$-AONT:
	\[
	\left(
	\begin{array}{ccccccccc}
		0  & 0  & 0  & 0  & 1  & 1  & 1  & 1  & 1 \\ 
		0  & 1  & 1  & 1  & 0  & 1  & 1  & 1  & 2 \\ 
		0  & 1  & 1  & 1  & 1  & 0  & 1  & 2  & 1 \\ 
		1  & 0  & 1  & 2  & 0  & 1  & 2  & 3  & 1 \\ 
		1  & 0  & 1  & 2  & 1  & 0  & 3  & 1  & 2 \\ 
		1  & 1  & 0  & 3  & 0  & 2  & 1  & 1  & 1 \\ 
		1  & 1  & 2  & 0  & 3  & 3  & 2  & 1  & 3 \\ 
		1  & 2  & 0  & 3  & 1  & 0  & 1  & 3  & 2 \\ 
		1  & 2  & 3  & 0  & 2  & 1  & 1  & 3  & 3 
	\end{array}
	\right) .
	\]
\end{Example}

\begin{Example}
		\label{A254}
	A linear $(2,5,11,4)$-AONT:
	\[
	\left(
	\begin{array}{ccccccccccc}
		0  & 0  & 0  & 0  & 0  & 1  & 1  & 1  & 1  & 1  & 1 \\ 
		0  & 0  & 0  & 0  & 1  & 1  & 1  & 1  & 1  & 1  & 1 \\ 
		0  & 1  & 1  & 1  & 0  & 0  & 1  & 1  & 1  & 1  & 2 \\ 
		0  & 1  & 1  & 1  & 1  & 1  & 0  & 1  & 1  & 2  & 3 \\ 
		1  & 0  & 1  & 2  & 0  & 1  & 0  & 1  & 2  & 3  & 1 \\ 
		1  & 0  & 1  & 2  & 1  & 0  & 1  & 1  & 3  & 2  & 1 \\ 
		1  & 1  & 0  & 3  & 0  & 1  & 1  & 2  & 3  & 1  & 3 \\ 
		1  & 1  & 2  & 0  & 3  & 0  & 2  & 1  & 3  & 3  & 2 \\ 
		1  & 2  & 0  & 3  & 3  & 0  & 0  & 1  & 2  & 3  & 3 \\ 
		1  & 2  & 3  & 0  & 2  & 1  & 0  & 2  & 1  & 3  & 3 \\ 
		1  & 3  & 2  & 1  & 1  & 2  & 1  & 0  & 3  & 2  & 3
	\end{array}
	\right) .
	\]\end{Example}

\begin{Example}
		\label{A235}
	A linear $(2,3,8,5)$-AONT:
	\[
	\left(
	\begin{array}{cccccccc}
		0  & 1  & 1  & 1  & 1  & 1  & 1  & 1 \\ 
		0  & 1  & 1  & 1  & 1  & 2  & 2  & 4 \\ 
		1  & 0  & 1  & 2  & 3  & 2  & 4  & 1 \\ 
		1  & 0  & 1  & 2  & 4  & 3  & 1  & 2 \\ 
		1  & 1  & 0  & 3  & 2  & 1  & 3  & 4 \\ 
		1  & 1  & 0  & 3  & 4  & 4  & 2  & 2 \\ 
		1  & 2  & 3  & 0  & 1  & 3  & 2  & 1 \\ 
		1  & 2  & 3  & 0  & 1  & 4  & 1  & 4
	\end{array}
	\right) .
	\]
\end{Example}

\begin{Example}
		\label{A245}
	A linear $(2,4,10,5)$-AONT:
	\[
	\left(
	\begin{array}{cccccccccc}
		0  & 0  & 0  & 0  & 0  & 0  & 1  & 1  & 1  & 1 \\ 
		0  & 1  & 1  & 1  & 1  & 1  & 0  & 1  & 1  & 1 \\ 
		0  & 1  & 1  & 1  & 1  & 4  & 1  & 0  & 1  & 2 \\ 
		1  & 0  & 1  & 2  & 3  & 1  & 0  & 1  & 4  & 2 \\ 
		1  & 0  & 1  & 2  & 3  & 4  & 1  & 0  & 2  & 1 \\ 
		1  & 1  & 0  & 3  & 4  & 2  & 0  & 1  & 2  & 2 \\ 
		1  & 1  & 0  & 3  & 4  & 3  & 2  & 2  & 1  & 4 \\ 
		1  & 2  & 3  & 0  & 1  & 1  & 1  & 4  & 3  & 4 \\ 
		1  & 2  & 3  & 0  & 1  & 4  & 4  & 3  & 1  & 2 \\ 
		1  & 3  & 4  & 1  & 0  & 2  & 2  & 2  & 2  & 3 
	\end{array}
	\right)
	\]
\end{Example}

\begin{Example}
	\label{A237}
	A linear $(2,3,8,7)$-AONT:
	\[
	\left(
	\begin{array}{ccccccccc}
		0  & 0  & 0  & 0  & 1  & 1  & 1  & 1 \\ 
		0  & 1  & 1  & 1  & 0  & 1  & 1  & 1 \\ 
		1  & 0  & 1  & 2  & 0  & 1  & 2  & 3 \\ 
		1  & 1  & 0  & 3  & 1  & 1  & 2  & 4 \\ 
		1  & 2  & 3  & 5  & 3  & 0  & 1  & 3 \\ 
		1  & 3  & 4  & 6  & 6  & 0  & 1  & 2 \\ 
		1  & 4  & 5  & 0  & 5  & 2  & 5  & 1 \\ 
		1  & 5  & 6  & 4  & 2  & 3  & 0  & 1
	\end{array}
	\right) .
	\]
\end{Example}

\end{document}